\documentclass[10pt]{article}
\usepackage{a4wide}
\usepackage{amsmath}
\usepackage{amssymb}
\usepackage{latexsym}
\usepackage[pdftex]{graphicx}
\usepackage{color}
\usepackage[mathscr]{euscript}
\usepackage{url}
\usepackage{float}
\usepackage{subfigure}
\usepackage{morefloats}
\usepackage{mathabx}
\usepackage{wasysym}
\usepackage{mathtools}
\usepackage{enumitem}
\usepackage{booktabs}

\newenvironment{proof}{\vspace{1ex}\noindent{\bf Proof:}}{\hspace*{\fill}$\blacksquare$\vspace{1ex}}
\newenvironment{proofof}[1]{\vspace{1ex}\noindent{\bf Proof of #1:}}{\hspace*{\fill}$\blacksquare$\vspace{1ex}}

\newtheorem{theorem}{Theorem}
\newtheorem{lemma} [theorem] {Lemma}

\newtheorem{definition} [theorem] {Definition}
\newtheorem{remark} [theorem] {Remark}

\newcommand{\eR}[0]{\ensuremath{ \mathbb R}}

\newcommand{\By}[2]{\overset{\mbox{\tiny{#1}}}{#2}}

\newcommand{\upd}{\mathrm{d}}

\newcommand{\upY}{\mathrm{Y}}

\title{Supporting document to the paper ``Logical limit laws for minor-closed classes of graphs"}

\author{Peter Heinig\thanks{Hamburg University of Technology, Hamburg, Germany. E-mail: \texttt{heinig@ma.tum.de}.
The author gratefully acknowledges the support of TUM Graduate School's Thematic Graduate Center TopMath
at \newline {Technische Universit\"at M\"unchen}.}
\and
Tobias M\"uller\thanks{Utrecht University, Utrecht, the Netherlands. E-mail: \texttt{t.muller@uu.nl}.}
\and
Anusch Taraz\thanks{Hamburg University of Technology, Hamburg, Germany. E-mail: \texttt{taraz@tuhh.de}.
The author was supported in part by DFG grant TA 319/2-2.}
}

\begin{document}

\maketitle

Here we provide a hand-checkable proof for Lemma 4.9 in the paper~\cite{ourpaper}.

\begin{definition}[{$B_0$, $B_2$; cf. \cite[p.~327]{GimenezNoy}}]
\label{r54trerewr5w4ytrere5t4rw43re}
We have to work with the following functions: 
\begin{enumerate}[label={\rm(\arabic{*})}]
\item\label{r435terew534w254fe43rew}
$B_0$ $=$ 
$\tfrac{(3t-1)^2(t+1)^6\log(t+1)}{512t^6}$ 
$-$ $\tfrac{(3t^4 - 16t^3 + 6t^2 -1)\log(3t+1)}{32 t^3}$ 
$-$ $\tfrac{(3t+1)^2(-t+1)^6 \log(2t+1)}{1024t^6}$ \\ 
$+$ $\tfrac14\log(t+3) - \tfrac12\log(t) - \tfrac38\log(16)$ 
$-$ $\tfrac{(217t^6 + 920t^5 + 972t^4 + 1436t^3 + 205t^2 - 172t + 6)(-t+1)^2}
{2048 t^4(3t+1)(t+3)}$ \quad , 
\item\label{dfg54ter54trdr4e5wdfsew43re}
$B_2$ $=$ $\tfrac{(-t+1)^3(3t-1)(3t+1)(t+1)^3\log(t+1)}{256 t^6}$ 
     $-$ $\tfrac{(-t+1)^3(3t+1)\log(3t+1)}{32t^3}$ 
     $+$ $\tfrac{(3t+1)^2(-t+1)^6\log(2t+1)}{512t^6}$ \\
     $+$ $\tfrac{(t-1)^4(185t^4+698t^3-217t^2-160t+6)}{1024 t^4(3t+1)(t+3)}$ \quad . 
\end{enumerate}
\end{definition}

\begin{definition}[{$h_1(t)$, $h_2(t)$}]
\label{fdtr56y4tr4543243ewrs}
For every $t\in (0,1)$ we define 
\begin{enumerate}[label={\rm(\arabic{*})}]
\item\label{r65ew54te32545354re43ew}
$h_1(t):=\tfrac{2t+1}{(3t+1)(-t+1)}$\quad , 
\item\label{5654tr543retr42rew52r4ewds}
$h_2(t):=-\tfrac{t^2 (-t+1) \left(5 t^2 + 36 t + 18\right)}{2(t+3)(2t+1)(3t+1)^2}$ \quad .
\end{enumerate}
\end{definition}

\begin{definition}[{$Y(t)$; cf. \cite[p.~310]{GimenezNoy}}]
\label{ret45terdfsre4t54tere4r3r4eq43eqww43rewds} 
For every $t\in (0,1)$, and with $h_1$ and $h_2$ 
as in Definition~\ref{fdtr56y4tr4543243ewrs}, 
we define $Y(t) := -1 + h_1(t)\ \exp(h_2(t))$. 
\end{definition}

\begin{lemma}\label{frete54trter46trer5tewre5te5454terfd45etrd}
The function $t\mapsto Y(t)$ is strictly monotone increasing in the open interval $(0,1)$. 
\end{lemma}
\begin{proof}
The derivative of $Y$ is 
\begin{equation}\label{re654terwer54ttre54trert54rte54trw543we}
\tfrac{\upd}{\upd t} Y(t) = 
\tfrac{3 t^2 (144 + 736 t + 1256 t^2 + 799 t^3 + 141 t^4 + t^5 - 5 t^6)}
{(2t+1)(3t+1)^4(t^2+2t-3)^2}\ \exp\left(
-\tfrac{t^2 (-t+1) \left(5 t^2 + 36 t + 18\right)}{2(t+3)(2t+1)(3t+1)^2}\right) \quad . 
\end{equation}
The exponential function being a strictly positive real number for any real argument, 
\eqref{re654terwer54ttre54trert54rte54trw543we} implies
\begin{equation}\label{tre54ytrrte546657rter645rtee54tr}
t>0 \quad \mathrm{and}\quad \tfrac{\upd}{\upd t} Y(t) > 0 
\qquad\Leftrightarrow\qquad 
5 t^6 <  t^5 + 141 t^4 + 799 t^3 + 1256 t^2 + 736 t + 144\quad , 
\end{equation}
the latter of which is obviously true since already $5t^6 < 144$ for every $0<t<1$. 
\end{proof}

\begin{lemma}\label{fdtr5eyrr5tytrew54ter456tr54r}
With $h_1$ and $h_2$ as in Definition~\ref{fdtr56y4tr4543243ewrs} we have 
\begin{enumerate}[label={\rm(\arabic{*})}]
\item\label{wer5434564et3564ter5} 
$2.0941746325 - 10^{-10} < h_1(0.6263716633 - 10^{-10}) < 2.0941746325 + 10^{-10}$ \quad , 
\item\label{re54tert54re4re24ew4e} 
$2.0941746335 - 10^{-10} < h_1(0.6263716633 + 10^{-10}) < 2.0941746335 + 10^{-10}$ \quad , 
\item\label{fdgret56e4trwe4r3ewe43rew453r4er54ed} 
$-0.0460123254 - 10^{-10} < h_2(0.6263716633 - 10^{-10}) < -0.0460123254 + 10^{-10}$ \quad , 
\item\label{ret54ytretr6yt5ryet564t564rt6754tr}
$-0.0460123253 - 10^{-10} < h_2(0.6263716633 + 10^{-10}) < -0.0460123253 + 10^{-10}$ \quad . 
\end{enumerate}
\end{lemma}
\begin{proof}
Checking these statements is left to the reader, who is advised to entrust 
\emph{this} entirely routine task to an electronic computer. 
The functions $h_1$ and $h_2$ being rational, the statements can be checked 
via exact computations with arbitrary long integers, 
a standard functionality of several computer algebra systems (note that to check 
\ref{fdgret56e4trwe4r3ewe43rew453r4er54ed} and \ref{ret54ytretr6yt5ryet564t564rt6754tr} 
one of course does not have to compute fractions, but one can rewrite 
\ref{fdgret56e4trwe4r3ewe43rew453r4er54ed} and \ref{ret54ytretr6yt5ryet564t564rt6754tr} 
as a statement about adding, subtracting and multiplying integers). 

Let us add that for reaching certainty about 
the equalities \ref{fdgret56e4trwe4r3ewe43rew453r4er54ed} 
and \ref{ret54ytretr6yt5ryet564t564rt6754tr}, the closest 
non-commercial automated alternatives to hand-evaluation 
seem to be some \texttt{C} libraries for arbitrary precision 
arithmetic, like \texttt{GMP} or \texttt{iRRAM}. 
According to \cite{MUELLERUHRHANsomestepsintoverificationofexactrealarithmetic}, 
the code in the \texttt{iRRAM} package itself is currently in the process of 
being formally verfied.  
\end{proof}

We now derive Taylor polynomials taylormade for our purposes 
(the approximation in \ref{w54tryfe54t565tr54trd} is designed to be used twice: 
both for the evaluations of $\exp$ within $\upY$, 
and later on for evaluations $\exp(-\tilde{\nu})$ with $\tilde{\nu}$ an approximation of $\nu$): 
\begin{lemma}[{some Taylor approximations to $\exp$}]
\label{rw5t4rswre54trt546tr54trde}
{\color{white} We have:}
\begin{enumerate}[label={\rm(\Roman{*})}]
\item\label{dft546rte54ee54564544}
for every $x\in(0.48,0.49)$, 
\begin{enumerate}[label={\rm(\arabic{*})}]
\item\label{rte5645trw54654ert5465465} 
$\left\lvert\exp(x) - \sum_{0\leq i\leq 11} \tfrac{x^i}{i!}\right\rvert
 < 0.11998784433\cdot 10^{-11}$ 
\item\label{fter46tyr54trer54tr54tr54ter} 
$0.39995948109\cdot 10^{-12} + \sum_{0\leq i\leq 11} \tfrac{x^i}{i!}
 < 
\exp(x)
 < 
0.11998784433\cdot 10^{-11} + \sum_{0\leq i\leq 11} \tfrac{x^i}{i!}$ 
\end{enumerate} 
\item\label{w54tryfe54t565tr54trd}
for every $x\in(-0.05,0)$, 
\begin{enumerate}[label={\rm(\arabic{*})}]
\item\label{ret54etrwe54rew453re3red} 
$\left\lvert\exp(x) - \sum_{0\leq i\leq 5} \tfrac{x^i}{i!}\right\rvert 
< 2.1701388889\cdot 10^{-11}$ 
\item\label{fdgwr54yterew43rew435re543rds} 
$1.0850694444\cdot 10^{-11} + \sum_{0\leq i\leq 5} \tfrac{x^i}{i!}
 < 
\exp(x)
 < 
2.1701388889\cdot 10^{-11} + \sum_{0\leq i\leq 5} \tfrac{x^i}{i!}$ 
\end{enumerate} 
\end{enumerate}
\end{lemma}
\begin{proof}
As to \ref{dft546rte54ee54564544}, 
we develop $\exp$ around\footnote{If we would develop $\exp$ around a rational 
number $x_0$ inside the interval we are interested in, we'd need fewer than eleven 
terms to achieve the desired accuracy (w.r.t. arithmetic with arbitrary elements of $\eR$). 
But we would then stray from our path to a set of `certificates' for 
the $p_i$-inequalities consisting of rational computations only: 
Taylor's theorem would require us to know $\exp(x_0)$ in order to 
compute the coefficients of the approximating polynomial. Since $\exp(x_0)$ 
is irrational for every rational $x_0$ (e.g., \cite{MR1641996}), another 
approximation would be necessary, resulting in additional complexity 
outweighing the gain in simplicity due to a lower-degree polynomial. Same for 
developing around an irrational number of the form $\log(x_0)$ with rational $x_0$ 
inside the respective intervals (which would keep the constant term rational yet 
necessitate approximations for what value to substitute into the variable). 
So developing around $0$ seems the only sensible choice for our purposes of 
deriving rational certificates. The price of the ease of evaluating the constant 
term $\exp(0)$ is a higher number of terms in order to `bend' the Taylor polynomial 
to within the required accuracy at points far from $0$.} $0$ and use Lagrange's error 
term for Taylor's theorem: for every $k$ and every $x\in(0,0.49)$ 
there exists $\xi_x\in(0,0.49)$ 
such that $\exp(x)=\sum_{0\leq i\leq k-1}\tfrac{x^i}{i!}+\tfrac{\exp(\xi_x)}{k!}x^k$. 
Because of $1 = \exp(0) < \exp(\xi_x) < \exp(0.49) < \exp(1) < 3$, we therefore know 
\begin{equation}\label{ret54rtrw5rw54tere54ert}
\tfrac{1}{k!} x^k < \exp(x) - \sum_{0\leq i\leq k-1}\tfrac{x^i}{i!} < \tfrac{3}{k!} x^k \quad ,  
\end{equation}
for every $x\in (0,0.49)$. In particular, 
\begin{equation}\label{erw5trefdrwerewrewds}
\left\lvert\exp(x) - \sum_{0\leq i\leq k-1}\tfrac{x^i}{i!}\right\rvert < \tfrac{3}{k!} x^k
\quad\text{for every $x\in (0,0.49)$} \quad .
\end{equation}

As for \ref{dft546rte54ee54564544}, we require $k$ to be large enough 
to have $\tfrac{3}{k!}x^k < 10^{-11}$ 
for every $x\in (0.48,0.49)$ $\subseteq$ $(0,0.49)$, i.e., we require $k$ 
to satisfy $\tfrac{3}{k!} 0.49^k < 10^{-11}$. The smallest such $k$ is $k=12$.  
Since $\tfrac{3}{12!} 0.49^{12}$ $<$ $0.11998784433\cdot 10^{-11}$ 
and $0.39995948109\cdot 10^{-12}$ $<$ $\tfrac{1}{12!} 0.49^{12}$, 
\eqref{ret54rtrw5rw54tere54ert} 
implies \ref{dft546rte54ee54564544}.\ref{fter46tyr54trer54tr54tr54ter}, 
and hence \ref{dft546rte54ee54564544}.\ref{rte5645trw54654ert5465465}. 

As for \ref{w54tryfe54t565tr54trd}, 
for every $x\in(-0.05,0)$, there exists $\xi_x\in(-0.05,0)$ such that 
$\exp(x) = \sum_{0\leq i\leq k-1} \tfrac{x^i}{i!} + \tfrac{\exp(\xi_x)}{k!}x^k$. 
Since $\tfrac12<\exp(-0.05)<\exp(\xi_x)<\exp(0)=1$, we know that for 
every even $k$, and any $x\in (-0.05,0)$ we have $x^k>0$ and 
\begin{equation}\label{6574ertw43e54er}
\tfrac{1}{2k!} x^k < \exp(x) - \sum_{0\leq i\leq k-1}\tfrac{x^i}{i!} 
< \tfrac{1}{k!} x^k \quad , 
\end{equation}
while for every odd $k$ and any $x\in (-0.05,0)$ we have $x^k<0$ and 
\begin{equation}\label{ret6r354t54tr56rte} 
\tfrac{1}{k!} x^k < \exp(x) - \sum_{0\leq i\leq k-1}\tfrac{x^i}{i!} < \tfrac{1}{2k!}x^k\quad . 
\end{equation}

In particular we now know that for every $k$ (of whatever parity) and any $x\in(-0.05,0)$, 
\begin{equation}\label{rew45etrw5354tred5trdscx}
\bigl\lvert\exp(x) - \sum_{0\leq i\leq k-1} \tfrac{x^i}{i!}\bigr\rvert
<
\tfrac{1}{k!} \lvert x\rvert^k \quad . 
\end{equation}
We require $k$ to be large enough to have $\tfrac{1}{k!} \lvert x\rvert^k < 10^{-10}$ 
for every $x\in(-0.05,0)$, i.e., we require $k$ to satisfy 
$\tfrac{1}{k!} 0.05^k < 10^{-10}$. The smallest such $k$ is $k=6$. 
Since $k=6$ is even, \eqref{6574ertw43e54er} together with 
$1.0850694444\cdot 10^{-11} < \tfrac12\tfrac{1}{6!} 0.05^6$ 
and $\tfrac{1}{6!} 0.05^6 < 2.1701388889\cdot 10^{-11}$ imply 
\ref{w54tryfe54t565tr54trd}.\ref{fdgwr54yterew43rew435re543rds}, 
and hence \ref{w54tryfe54t565tr54trd}.\ref{ret54etrwe54rew453re3red}. 
In particular we know that $\sum_{0\leq i\leq 5} \tfrac{x^i}{i!}$ underestimates $\exp(x)$ 
for every $x\in(-0.05,0)$. 
\end{proof}

\begin{lemma}[{verified bounds for $t_0$}]
\label{t54treert4re354tree34terfd}
There exists exactly one real number $t_0\in(0,1)$ with $Y(t_0)=1$, 
and it satisfies 
\begin{equation}\label{fregtyess5rtyewe54retfews5t} 
0.6263716633 - 10^{-10} < t_0 < 0.6263716633 + 10^{-10} \quad . 
\end{equation}
\end{lemma}
\begin{proof}
Since all factors in denominators within $Y(t)$ are non-zero 
for $t\in(0,1)$, the function $t\mapsto Y(t)$ is continuous as a composition of continuous 
functions. By Lemma~\ref{frete54trter46trer5tewre5te5454terfd45etrd}, it is moreover 
strictly monotone increasing in $(0,1)$. Therefore the claim follows (existence from 
continuity, uniqueness from monotonicity) via the Intermediate Value Theorem 
if we can show that 
\begin{enumerate}[label={\rm(\arabic{*})}]
\item\label{gty56rtet546r6uyther54terdf} 
$Y(0.6263716633 - 10^{-10}) < 1$ \quad , 
\item\label{dsffr5t6terewrtewr5yt4re5t4rd} 
$Y(0.6263716633 + 10^{-10}) > 1$ \quad .
\end{enumerate}
A finite certificate for \ref{gty56rtet546r6uyther54terdf} is given by the calculation 
\begin{align}\label{ewr5464535342543454r54rtds}
Y(0.6263716633 - 10^{-10}) 
& =  - 1 + h_1(0.6263716633 - 10^{-10})\cdot\exp(h_2(0.6263716633 - 10^{-10})) \notag \\
\parbox{0.3\linewidth}
{(by the upper bounds in 
\ref{wer5434564et3564ter5} and \ref{fdgret56e4trwe4r3ewe43rew453r4er54ed} 
in Lemma~\ref{fdtr5eyrr5tytrew54ter456tr54r}, 
and since $\exp$ is monotone increasing)} 
& < - 1 + 2.0941746326\cdot\exp(-0.0460123253) \notag \\
\parbox{0.3\linewidth}
{(by the upper bound in \ref{ret54etrwe54rew453re3red})} 
& < - 1 + 2.0941746326\cdot \notag \\
& \left(2.1701388889\cdot 10^{-11} + \sum_{0\leq i\leq 5} \tfrac{(-0.0460123253)^i}{i!}\right)
\notag \\
& = 0.999999999554440826331073832451\backslash \notag \\
&{\color{white} =}\ 82705870208185832244853853496068 + \tfrac13\cdot 10^{-62} < 1 \quad  , 
\end{align}
while a finite certificate for \ref{dsffr5t6terewrtewr5yt4re5t4rd} 
is given by the calculation 
\begin{align}\label{w564terer5462436e34e5tr54ewds}
Y(0.6263716633 + 10^{-10}) 
& = -1 + h_1(0.6263716633 + 10^{-10})\cdot\exp(h_2(0.6263716633 + 10^{-10})) \notag \\
\parbox{0.3\linewidth}
{(by the lower bounds in 
\ref{re54tert54re4re24ew4e} and \ref{ret54ytretr6yt5ryet564t564rt6754tr} 
in Lemma~\ref{fdtr5eyrr5tytrew54ter456tr54r}, 
and since $\exp$ is monotone increasing)} 
& > - 1 + 2.0941746334\cdot\exp(-0.0460123254) \notag \\ 
\parbox{0.3\linewidth}
{(by the lower bound in \ref{fdgwr54yterew43rew435re543rds})} 
& > -1 + 2.0941746334\cdot \notag \\
& \left(1.0850694444\cdot 10^{-11} + \sum_{0\leq i\leq 5} \tfrac{(-0.0460123254)^i}{i!}\right)
\notag \\
& = 1.0000000000957417297668951405800\backslash \notag \\
&{\color{white} =}\ 480697033915364640304336242832 > 1\ , 
\end{align}
where in each case $\backslash$ denotes that a number contiguously continues 
in the next line. 
\end{proof}

The following defines the function $t$ from \cite{GimenezNoy}, 
with explicit values for the `suitable small neighborhood of $1$' 
\cite[p.~317, paragraph~2]{GimenezNoy}: 
\begin{definition}
\label{rew54terer5etwe5redew54re}
For every $y\in (0.9999999996,1.00000000009)$ we define $t(y)$ 
to be the unique $t\in(0.6263716633 - 10^{-10},0.6263716633 + 10^{-10})$ with $Y(t)=y$. 
\end{definition}
Let us note that $t_0=t(1)$. 

\begin{remark}[{correctness of Definition~\ref{rew54terer5etwe5redew54re}}]
Definition~\ref{rew54terer5etwe5redew54re} does indeed define a function 
\begin{equation}\label{ret56trrte654e65tref}
t\colon (0.9999999996,1.00000000009)\to(0.6263716633 - 10^{-10},0.6263716633 + 10^{-10})\quad . 
\end{equation}
\end{remark}
\begin{proof}
Uniqueness of the $t(y)$ from Definition~\ref{rew54terer5etwe5redew54re} follows from 
Lemma~\ref{frete54trter46trer5tewre5te5454terfd45etrd}, while for existence we have 
to show that the argument in the proof of Lemma~\ref{t54treert4re354tree34terfd} can be 
carried out with any $y\in(0.9999999996,1.00000000009)$ replacing the $1$ in the 
conditions \ref{gty56rtet546r6uyther54terdf} and \ref{dsffr5t6terewrtewr5yt4re5t4rd} 
of Lemma~\ref{t54treert4re354tree34terfd}. 
This follows from \eqref{ewr5464535342543454r54rtds} and \eqref{w564terer5462436e34e5tr54ewds}: 
since \newline
$0.99999999955444082633107383245182705870208185832244853853496068$ $+$ $\tfrac13\cdot 10^{-62}$ 
\newline $<$ $0.9999999996$ and \newline 
$1.0000000000957417297668951405800480697033915364640304336242832$ 
$>$ $1.00000000009$, \newline 
each of these calculations can be used as is for proving 
the existence of any $t(y)$ with $y\in(0.9999999996,1.00000000009)$. 
\end{proof}

\begin{definition}[{$R$; cf. \cite[(2.6)]{GimenezNoy}}]
\label{ftr5y4terwr454et54tre43ew354er54rted}
With $t$ as in Definition~\ref{rew54terer5etwe5redew54re}, we define the function 
\begin{align}\label{ewr54ertewtrwr54erfdx}
R\colon (0.9999999996,1.00000000009) & \longrightarrow \eR \notag \\
y & \longmapsto R(y) := \frac{\left(3\cdot t(y) + 1\right)
\left(-t(y) + 1\right)^3}{16\cdot t(y)^3} \quad . 
\end{align}
\end{definition}

\begin{lemma}\label{fdg54treert3254rede54terdrw45r657ty}
With $\xi(t):=\tfrac{\left(3\cdot t + 1\right)\left(-t + 1\right)^3}{16\cdot t^3}$, 
\begin{enumerate}[label={\rm(\arabic{*})}]
\item\label{sdwferf54trre5665r45tere64tr} 
$0.03819109771 < \xi(0.6263716633 - 10^{-10}) < 0.03819109772$ \quad , 
\item\label{fdgwret4rgddewr54etrderr54te}
$0.03819109762 < \xi(0.6263716633 + 10^{-10}) < 0.03819109763$ \quad . 
\end{enumerate}
\end{lemma}
\begin{proof}
Finite statements about integers. Same comments as in the proof of 
Lemma~\ref{fdtr5eyrr5tytrew54ter456tr54r} apply. 
\end{proof}

\begin{lemma}[{some pointwise bounds for $B_0(t)$}]
\label{45ewrtewswrertewdsxz}
With $B_0$ as in Definition~\ref{r54trerewr5w4ytrere5t4rw43re}.\ref{r435terew534w254fe43rew}, 
\begin{enumerate}[label={\rm(\arabic{*})}]
\item\label{fe543e5456454657trt54rewds54tr} 
$0.00073969957 < B_0(0.6263716633 - 10^{-10}) < 0.00073969958$ \quad , 
\item\label{r56t65t5ret4ree5y4tre54tre54tr}
$0.00073969956 < B_0(0.6263716633 + 10^{-10}) < 0.00073969957$ \quad . 
\end{enumerate}
\end{lemma}
\begin{proof}
Finite statements about integers. The same comments as in the proof of 
Lemma~\ref{fdtr5eyrr5tytrew54ter456tr54r} apply.
\end{proof}

\begin{lemma}[{uniform bounds for $B_0(t)$}]
\label{fdert34w43r4wer5t433453rte54tref}
With $B_0$ as in 
Definition~\ref{r54trerewr5w4ytrere5t4rw43re}.\ref{r435terew534w254fe43rew}, 
\begin{equation}\label{fde54tyreewdsdse4r4ert5463treer5t4ref} 
0.00073969896 < B_0(t) < 0.00073970019 
\end{equation}
for every $t\in I:=(0.6263716633 - 10^{-10},0.6263716633 + 10^{-10})$.
\end{lemma}
\begin{proof}
If we had a proof that $B_0$ is monotone decreasing in $I$, 
then \eqref{fde54tyreewdsdse4r4ert5463treer5t4ref} would follow from the slightly stronger 
pointwise bounds in Lemma~\ref{45ewrtewswrertewdsxz}---but the (known) continuity of $B_0$ 
alone is of course not enough to use Lemma~\ref{45ewrtewswrertewdsxz}. 
Unfortunately, a complete proof of this monotonicity seems to require 
at least as much work as the proof of \eqref{fde54tyreewdsdse4r4ert5463treer5t4ref} 
that follows. 

The plan of the proof is the following: for each of the seven summands in $B_0$ we will 
derive both upper and lower bounds which uniformly hold in $I$. In the end, we add these 
bounds to derive the bounds in \eqref{fde54tyreewdsdse4r4ert5463treer5t4ref}. 

In the following paragraph, we prove the uniform bounds 
\begin{equation}\label{fdgert4tre4564674e5434eere43ter}
0.22495616614 < \tfrac{(3t-1)^2(t+1)^6\log(t+1)}{512t^6} < 0.22495616711 \quad 
\text{for every $t\in I$} \quad . 
\end{equation}

Since $3\cdot t > 1$ for every $t\in I$, 
the function $t\mapsto (3t-1)^2$ is evidently monotone increasing in $I$. 
So are the two functions $t\mapsto(t+1)^6$ and $t\mapsto\log(t+1)$. 
Therefore, $t\mapsto (3t-1)^2(t+1)^6\log(t+1)$ is monotone increasing in $I$ 
as a product of three such functions. Hence, for every $t\in I$, 
\begin{equation}\label{4533234}
(3t-1)^2(t+1)^6\log(t+1) 
< 
(3t-1)^2(t+1)^6\log(t+1)\biggr\rvert_{t=0.6263716633 + 10^{-10}}
< 
6.95601448698
\end{equation}
and 
\begin{equation}\label{f6574te54rd}
(3t-1)^2(t+1)^6\log(t+1) 
> 
(3t-1)^2(t+1)^6\log(t+1)\biggr\rvert_{t=0.6263716633 - 10^{-10}}
> 
6.95601447059 \quad .
\end{equation}
The function $t\mapsto 512 t^6$ is evidently monotone increasing in $I$. 
Hence, for every $t\in I$, 
\begin{equation}\label{fdgter54tyreew54te5w324r5653re}
512 t^6 > 512 t^6\biggr\rvert_{t=0.6263716633 - 10^{-10}} > 30.92164387643 
\end{equation}
and 
\begin{equation}\label{sfrt563etrdgfss3w456435tr54re}
512 t^6 < 512 t^6\biggr\rvert_{t=0.6263716633 + 10^{-10}} < 30.92164393568 \quad .
\end{equation}
Since \eqref{f6574te54rd} and \eqref{sfrt563etrdgfss3w456435tr54re} hold in all of $I$, 
it follows that, for every $t\in I$, 
\begin{equation}\label{fdtre6y5teretr4er54tr54tr}
\tfrac{(3t-1)^2(t+1)^6\log(t+1)}{512t^6} 
> 
\tfrac{6.95601447059}{30.92164393568} 
> 
0.22495616614 \quad ,
\end{equation}
proving the lower bound in \eqref{fdgert4tre4564674e5434eere43ter}. 

Since \eqref{4533234} and \eqref{fdgter54tyreew54te5w324r5653re} hold in all of $I$, 
it follows that, for every $t\in I$, 
\begin{equation}\label{ftr56ertew432w3e543r43w5wrerw54ewds} 
\tfrac{(3t-1)^2(t+1)^6\log(t+1)}{512t^6} 
< 
\tfrac{6.95601448698}{30.92164387643} 
< 
0.22495616711\quad , 
\end{equation}
proving the upper bound in \eqref{fdgert4tre4564674e5434eere43ter}. 

In the following paragraph, we prove the uniform bounds 
\begin{equation}\label{rew54rse43rwe24w3ew}
-0.28456395530 < \tfrac{(3t^4 - 16t^3 + 6t^2 -1)\log(3t+1)}{32 t^3} < -0.28456395528 \quad 
\text{for every $t\in I$} \quad . 
\end{equation}

Since $2+\sqrt{3} > 1$, $2-\sqrt{3}<0.5$ and 
$\tfrac{\upd}{\upd t}(12t^3 - 48t^2 + 12t)$ $=$ $36t^2 - 96t + 12$ $=$ 
$12 t (t-(2+\sqrt{3})) (t-(2-\sqrt{3}))$, it is evident that 
$\tfrac{\upd}{\upd t}(12t^3 - 48t^2 + 12t) < 0$ for every $t\in I$, i.e., 
$t\mapsto 3t^4-16t^3+6t^2-1$ is strictly monotone decreasing in $I$, so 
\begin{align}\label{w54tewe4332ewe54we354tre543terfre635ters}
    3t^4-16t^3+6t^2-1 
& > 3t^4-16t^3+6t^2-1\biggr\rvert_{t=0.6263716633 + 10^{-10}} \notag \\
& = -2.1161809442159711262496568523624448554192\quad\text{for every $t\in I$} \quad , 
\end{align}
and
\begin{align}\label{rte54trwe43ew54ewreew43rewweerwds}
    3t^4-16t^3+6t^2-1 
& < 3t^4-16t^3+6t^2-1\biggr\rvert_{t=0.6263716633 - 10^{-10}} \notag \\
& = -2.1161809425425888723475949656101944348672\quad\text{for every $t\in I$} \quad . 
\end{align}
The function $t\mapsto \log(3t+1)$ is evidently strictly monotone increasing in $I$, hence 
\begin{equation}\label{wr54tww2er54r43e5345r}
\log(3t+1) > \log(3t+1)\biggr\rvert_{t=0.6263716633 - 10^{-10}} 
> 1.05748295164 \quad\text{for every $t\in I$} \quad ,
\end{equation}
and
\begin{equation}\label{rt546tregfstrertre54trw43er}
\log(3t+1) < \log(3t+1)\biggr\rvert_{t=0.6263716633 + 10^{-10}} 
< 1.05748295186 \quad\text{for every $t\in I$} \quad .
\end{equation}
The function $t\mapsto 32 t^3$ is evidently monotone increasing in $I$. 
Hence, for every $t\in I$, 
\begin{equation}\label{de4w4wr3e43reww43resa}
32t^3 > 32t^3\biggr\rvert_{t=0.6263716633 - 10^{-10}} = 7.864050340179393384432870014976 \quad ,
\end{equation}
and
\begin{equation}\label{ew54tre3453543r4tewrfds}
32t^3 < 32t^3\biggr\rvert_{t=0.6263716633 + 10^{-10}} = 7.864050347712349427668874499328 \quad . 
\end{equation}
It follows that, for every $t\in I$, 
\begin{align}\label{er54tere4343254r54tref}
  & \tfrac{(3t^4-16t^3+6t^2-1)\log(3t+1)}{32t^3}\notag \\
\parbox{0.25\linewidth}{(by \eqref{w54tewe4332ewe54we354tre543terfre635ters})}
> & \tfrac{(-2.1161809442159711262496568523624448554192)\cdot\log(3t+1)}{32t^3} \notag \\
\parbox{0.25\linewidth}{(by \eqref{rt546tregfstrertre54trw43er}; 
we recall that multiplying with a negative number flips an inequality)}
> & \tfrac{(-2.1161809442159711262496568523624448554192)\cdot 1.05748295186}{32t^3} \notag \\
\parbox{0.25\linewidth}{(by \eqref{ew54tre3453543r4tewrfds})}
> & \tfrac{(-2.1161809442159711262496568523624448554192)\cdot 1.05748295186}
{7.864050347712349427668874499328} \notag \\
> & -0.28456395530\quad ,
\end{align}
proving the lower bound in \eqref{rew54rse43rwe24w3ew}, and also that, for every $t\in I$, 
\begin{align}\label{trr65rte43e5ter234ew}
  & \tfrac{(3t^4-16t^3+6t^2-1)\log(3t+1)}{32t^3}\notag \\
\parbox{0.25\linewidth}{(by \eqref{rte54trwe43ew54ewreew43rewweerwds})}
< & \tfrac{(-2.1161809425425888723475949656101944348672)\cdot\log(3t+1)}{32t^3} \notag \\
\parbox{0.25\linewidth}{(by \eqref{wr54tww2er54r43e5345r}; 
we recall that multiplying with a negative number flips an inequality)}
< & \tfrac{(-2.1161809425425888723475949656101944348672)\cdot 1.05748295164}{32t^3} \notag \\
\parbox{0.25\linewidth}{(by \eqref{de4w4wr3e43reww43resa})}
< & \tfrac{(-2.1161809425425888723475949656101944348672)\cdot 1.05748295164}
{7.864050340179393384432870014976} \notag \\
< & -0.28456395528\quad ,
\end{align}
which proves the upper bound in \eqref{rew54rse43rwe24w3ew}. 

In the following paragraph, we prove the uniform bounds 
\begin{equation}\label{dewr43eqwe43rew4rewds}
0.00029614190 < \tfrac{(3t+1)^2(-t+1)^6\log(2t+1)}{1024 t^6} < 0.00029614191 \quad 
\text{for every $t\in I$} \quad . 
\end{equation}
While it is evident that $t\mapsto(3t+1)^2$ is strictly monotone increasing, 
and $t\mapsto (-t+1)^6$ strictly monotone decreasing in $I$, it is not evident 
whether the product $t\mapsto(3t+1)^2(-t+1)^6$ decreases or increases in $I$. 
To decide this, we note that 
$\tfrac{\upd}{\upd t}(3t+1)^2(-t+1)^6$ $=$ $24 (-1+t)^5 t (1 + 3 t)$, 
and from this factorization it \emph{is} evident 
that $\tfrac{\upd}{\upd t}(3t+1)^2(-t+1)^6 < 0$ for every $t\in I$, 
hence that $t\mapsto(3t+1)^2(-t+1)^6$ is indeed strictly monotone decreasing in $I$. 
Therefore, 
\begin{align}\label{fdr54tewer24w3re32r3ewr54re43}
(3t+1)^2(-t+1)^6 & < (3t+1)^2(-t+1)^6\biggr\rvert_{t=0.6263716633 - 10^{-10}} \notag \\
& < 0.02255053559 \quad\text{for every $t\in I$}\quad , 
\end{align}
and 
\begin{align}\label{fwre54terrt5y4rt54354weer54tr}
(3t+1)^2(-t+1)^6 & > (3t+1)^2(-t+1)^6\biggr\rvert_{t=0.6263716633 + 10^{-10}} \notag \\
& > 0.02255053553 \quad\text{for every $t\in I$}\quad . 
\end{align}
Moreover, since function $t\mapsto\log(2t+1)$ evidently 
is strictly monotone increasing in $I$, we know that 
\begin{align}\label{dew54retw54trt5e4tyrgfd}
\log(2t+1) & > \log(2t+1)\biggr\rvert_{t=0.6263716633 - 10^{-10}} \notag \\
& > 0.81214872970 \quad\text{for every $t\in I$}\quad , 
\end{align}
and 
\begin{align}\label{d5645tw54te54erte5w4tref}
\log(2t+1) & < \log(2t+1)\biggr\rvert_{t=0.6263716633 + 10^{-10}} \notag \\
& < 0.81214872989 \quad\text{for every $t\in I$}\quad . 
\end{align}
Furthermore, since the function $t\mapsto 1024t^6$ evidently 
is strictly monotone increasing in $I$, we know that 
\begin{align}\label{ew453reew34req434wre54rew}
1024t^6 & > 1024t^6\biggr\rvert_{t=0.6263716633 - 10^{-10}} \notag \\
& > 61.84328775287 \quad\text{for every $t\in I$}\quad , 
\end{align}
and 
\begin{align}\label{43teww43r4ter5e4tre43rew3}
1024t^6 & < 1024t^6\biggr\rvert_{t=0.6263716633 + 10^{-10}} \notag \\
& < 61.84328787136 \quad\text{for every $t\in I$}\quad . 
\end{align}
It follows that, for every $t\in I$, 
\begin{align}\label{w54rteew45654w2564654yt65tre}
  & \tfrac{(3t+1)^2(-t+1)^6\log(2t+1)}{1024 t^6} \notag \\
\parbox{0.25\linewidth}{(by \eqref{fwre54terrt5y4rt54354weer54tr})}
> & \tfrac{0.02255053553\cdot\log(2t+1)}{1024 t^6} \notag \\
\parbox{0.25\linewidth}{(by \eqref{dew54retw54trt5e4tyrgfd})}
> & \tfrac{0.02255053553\cdot 0.81214872970}{1024 t^6} \notag \\
\parbox{0.25\linewidth}{(by \eqref{43teww43r4ter5e4tre43rew3})}
> & \tfrac{0.02255053553\cdot 0.81214872970}{61.84328787136} \notag \\
> & 0.00029614190\quad ,
\end{align}
proving the lower bound in \eqref{dewr43eqwe43rew4rewds}, 
and also that, for every $t\in I$, 
\begin{align}\label{w324w4e45354543543erd43red}
  & \tfrac{(3t+1)^2(-t+1)^6\log(2t+1)}{1024 t^6} \notag \\
\parbox{0.25\linewidth}{(by \eqref{fdr54tewer24w3re32r3ewr54re43})}
< & \tfrac{0.02255053559\cdot\log(2t+1)}{1024 t^6} \notag \\
\parbox{0.25\linewidth}{(by \eqref{d5645tw54te54erte5w4tref})}
< & \tfrac{0.02255053559\cdot 0.81214872989}{1024 t^6} \notag \\
\parbox{0.25\linewidth}{(by \eqref{ew453reew34req434wre54rew})}
< & \tfrac{0.02255053559\cdot 0.81214872989}{61.84328775287} \notag \\
< & 0.00029614191\quad ,
\end{align}
which proves the upper bound in \eqref{dewr43eqwe43rew4rewds}.

Since $t\mapsto\tfrac14\log(t+3)$ is evidently strictly monotone increasing in $I$, 
we know that, for every $t\in I$, 
\begin{align}\label{dwr4t3ree43re43re4545ttyrr}
\tfrac14\log(t+3) & > \tfrac14\log(t+3)\biggr\rvert_{t=0.6263716633 - 10^{-10}} \notag \\
& > 0.32205815164\quad \text{for every $t\in I$} \quad , 
\end{align}
and
\begin{align}\label{fdgwr5t4ewrfr54t3er46e4rte}
\tfrac14\log(t+3) & < \tfrac14\log(t+3)\biggr\rvert_{t=0.6263716633 + 10^{-10}} \notag \\
& < 0.32205815165\quad 
\text{for every $t\in I$} \quad . 
\end{align}

Since $t\mapsto\tfrac12\log(t)$ is evidently strictly monotone increasing in $I$, 
we know that, for every $t\in I$, 
\begin{align}\label{e6t5rw54rter54tr564}
\tfrac12\log(t) & > \tfrac12\log(t)\biggr\rvert_{t=0.6263716633 - 10^{-10}} \notag \\
& > -0.23390568644 \quad \text{for every $t\in I$} \quad , 
\end{align}
and
\begin{align}\label{rt54454rterrw5454re}
\tfrac12\log(t) & < \tfrac12\log(t)\biggr\rvert_{t=0.6263716633 + 10^{-10}} \notag \\
& < -0.23390568627\quad \text{for every $t\in I$} \quad . 
\end{align}

As to the summand $\tfrac38\log(16)$ in $B_0$, there are the bounds 
\begin{equation}\label{fert5e4tresd54twrte546tere54treds}
1.03972077083< \tfrac38\log(16) < 1.03972077084\quad . 
\end{equation}

In the following paragraph, we prove the uniform bounds 
\begin{equation}\label{fewr54terew5t4ree54tree543ertf}
0.02472734758 < \tfrac{(217t^6 + 920t^5 + 972t^4 + 1436t^3 + 205t^2 - 172t + 6)(-t+1)^2}
{2048 t^4(3t+1)(t+3)} < 0.02472734762 \quad 
\text{for every $t\in I$} \quad . 
\end{equation}
We have $\tfrac{\upd}{\upd t}$ 
$(217t^6 + 920t^5 + 972t^4 + 1436t^3 + 205t^2 - 172t + 6)(-t+1)^2$ $=$ 
$2(t-1)(868t^6 + 2569t^5 + 616t^4 + 1646t^3 - 1744t^2 - 463t + 92)$, 
and since $2(t-1)<0$ for every $t\in I$, to prove that 
$t\mapsto (217t^6 + 920t^5 + 972t^4 + 1436t^3 + 205t^2 - 172t + 6)(-t+1)^2$ 
is strictly monotone increasing in $I$ it suffices to show that 
$868t^6 + 2569t^5 + 616t^4 + 1646t^3 - 1744t^2 - 463t + 92 < 0$ for every $t\in I$. 
This is equivalent to 
\begin{equation}\label{er563r4eewe54w53456trerw53r4ew}
868t^6 + 2569t^5 + 616t^4 + 1646t^3 + 92\quad <\quad 1744t^2 + 463t\quad 
\text{for every $t\in I$}\quad . 
\end{equation}
Since both $t\mapsto 868t^6 + 2569t^5 + 616t^4 + 1646t^3 + 92$ and $t\mapsto 1744t^2 + 463t$, 
are strictly monotone increasing in $I$, we have, for every $t\in I$, 
\begin{align}\label{ft5e4treew5r4tere54tre45456r54y5542ew654}
868t^6 + 2569t^5 + 616t^4 + 1646t^3 + 92 
& < 868t^6 + 2569t^5 + 616t^4 + 1646t^3 + 92\biggr\rvert_{t=0.6263716633 + 10^{-10}}\notag \\
& < 891.450148292474 \notag \\
& < 974.25358710372530451456 \notag \\
& = 1744t^2 + 463t\biggr\rvert_{t=0.6263716633 - 10^{-10}} \quad < \quad 1744t^2 + 463t\quad , 
\end{align}
proving \eqref{er563r4eewe54w53456trerw53r4ew}. Since we now know that 
$t\mapsto (217t^6 + 920t^5 + 972t^4 + 1436t^3 + 205t^2 - 172t + 6)(-t+1)^2$ is 
strictly monotone increasing in $I$, it follows that, for every $t\in I$, 
\begin{align}\label{fretsfe5yt4retrter543reew4t3re}
  & (217t^6+920t^5+972t^4+1436t^3+205t^2-172t+6)(-t+1)^2\notag\\ 
> & (217t^6+920t^5+972t^4+1436t^3+205t^2-172t+6)(-t+1)^2\biggr\rvert_{t=0.6263716633 - 10^{-10}}\notag\\ 
> & 81.3892822256
\end{align}
and 
\begin{align}\label{ertw54yr54654tre5434re435e}
  & (217t^6+920t^5+972t^4+1436t^3+205t^2-172t+6)(-t+1)^2\notag\\ 
< & (217t^6+920t^5+972t^4+1436t^3+205t^2-172t+6)(-t+1)^2\biggr\rvert_{t=0.6263716633 + 10^{-10}}\notag\\
< & 81.3892822381 \quad . 
\end{align}
Since $t\mapsto 2048 t^4(3t+1)(t+3)$ is evidently strictly monotone increasing in $I$, 
it follows that, for every $t\in I$, 
\begin{align}\label{fd5e4ytrr65t4tr423e43reds}
2048 t^4(3t+1)(t+3) > 2048 t^4(3t+1)(t+3)\biggr\rvert_{t=0.6263716633 - 10^{-10}} > 3291.4683555 
\end{align}
and 
\begin{align}\label{er654re42535454trry5t4rry5e4t54}
2048 t^4(3t+1)(t+3) < 2048 t^4(3t+1)(t+3)\biggr\rvert_{t=0.6263716633 + 10^{-10}} < 3291.4683606\quad.
\end{align}

It follows that, for every $t\in I$, 
\begin{align}\label{dr54tre45454rte43res}
  & \tfrac{(217t^6 + 920t^5 + 972t^4 + 1436t^3 + 205t^2 - 172t + 6)(-t+1)^2}
{2048 t^4(3t+1)(t+3)} \notag \\
\parbox{0.25\linewidth}{(by \eqref{fretsfe5yt4retrter543reew4t3re})}
> & \tfrac{81.3892822256}{2048 t^4(3t+1)(t+3)} \notag \\
\parbox{0.25\linewidth}{(by \eqref{er654re42535454trry5t4rry5e4t54})}
> & \tfrac{81.3892822256}{3291.4683606} > 0.02472734758\quad , 
\end{align}
proving the lower bound in \eqref{fewr54terew5t4ree54tree543ertf}, and 
\begin{align}\label{ewr4teq435353465tr4543ewsq2w}
  & \tfrac{(217t^6 + 920t^5 + 972t^4 + 1436t^3 + 205t^2 - 172t + 6)(-t+1)^2}
{2048 t^4(3t+1)(t+3)} \notag \\
\parbox{0.25\linewidth}{(by \eqref{ertw54yr54654tre5434re435e})}
< & \tfrac{81.3892822381}{2048 t^4(3t+1)(t+3)} \notag \\
\parbox{0.25\linewidth}{(by \eqref{fd5e4ytrr65t4tr423e43reds})}
< & \tfrac{81.3892822381}{3291.4683555} < 0.02472734762\quad , 
\end{align}
proving the upper bound in \eqref{fewr54terew5t4ree54tree543ertf}.

We now add our uniform bounds for the summands in $B_0$ to prove the 
uniform bounds in \eqref{fde54tyreewdsdse4r4ert5463treer5t4ref}. In 
doing so, we have to pay attention which summand appears 
with a minus-sign in the definition of $B_0$. 

From the lower bound in \eqref{fdgert4tre4564674e5434eere43ter}, 
the upper bounds in \eqref{rew54rse43rwe24w3ew} and \eqref{dewr43eqwe43rew4rewds}, 
the lower bound in \eqref{dwr4t3ree43re43re4545ttyrr}, and 
the upper bounds in \eqref{e6t5rw54rter54tr564}, 
\eqref{fert5e4tresd54twrte546tere54treds} and \eqref{fewr54terew5t4ree54tree543ertf}, 
it follows that, for every $t\in I$, 
\begin{align}\label{dfe54terswe54e43r5e445e4tr}
B_0(t) & > 0.22495616614 - (-0.28456395528) - 0.00029614191 \notag \\ 
     & + 0.32205815164 - (-0.23390568627) - (1.03972077084) - (0.02472734762) \notag \\
     & = 0.00073969896 \quad , 
\end{align}
proving the lower bound in \eqref{fde54tyreewdsdse4r4ert5463treer5t4ref}. 

From the upper bound in \eqref{fdgert4tre4564674e5434eere43ter}, 
the lower bounds in \eqref{rew54rse43rwe24w3ew} and \eqref{dewr43eqwe43rew4rewds}, 
the upper bound in \eqref{dwr4t3ree43re43re4545ttyrr} and 
the lower bounds in \eqref{e6t5rw54rter54tr564}, 
\eqref{fert5e4tresd54twrte546tere54treds} and \eqref{fewr54terew5t4ree54tree543ertf} 
it follows that, for every $t\in I$, 
\begin{align}\label{fdrt5tsw4tr3e453rtewr65tr4e}
B_0(t) & < 0.22495616711 - (-0.28456395530) - 0.00029614190 \notag \\
     & + 0.32205815165 - (-0.23390568644) - (1.03972077083) - (0.02472734758) \notag \\ 
     & = 0.00073970019 \quad , 
\end{align}
proving the upper bound in \eqref{fde54tyreewdsdse4r4ert5463treer5t4ref}. 
This completes the proof of Lemma~\ref{fdert34w43r4wer5t433453rte54tref}. 
\end{proof}

\begin{lemma}[{bounds for $B_0(t_0)$}]
\label{fdrte54rdere4rewr544rt}
With $B_0$ as in 
Definition~\ref{r54trerewr5w4ytrere5t4rw43re}.\ref{r435terew534w254fe43rew}, 
\begin{equation}\label{fwr54trgfewr5etew54rewds}
0.00073969896 < B_0(t_0) < 0.00073970019
\end{equation}
\end{lemma}
\begin{proof}
In view of Lemma~\ref{t54treert4re354tree34terfd}, 
the bounds in \eqref{fwr54trgfewr5etew54rewds} follow 
from the uniform bounds in Lemma~\ref{fdert34w43r4wer5t433453rte54tref}.
\end{proof}

\begin{lemma}[{pointwise bounds for $B_2(t_0)$}]
\label{re54ter354etr5345trd}
With $B_2$ as in 
Definition~\ref{r54trerewr5w4ytrere5t4rw43re}.\ref{dfg54ter54trdr4e5wdfsew43re}, 
\begin{enumerate}[label={\rm(\arabic{*})}]
\item\label{dser4564ttreew43refd} 
$-0.0014914312 - 10^{-10} < B_2(0.6263716633 - 10^{-10}) < -0.0014914312 + 10^{-10}$ \quad , 
\item\label{rt563243ew545564t524er}
$-0.0014914312 - 10^{-10} < B_2(0.6263716633 + 10^{-10}) < -0.0014914312 + 10^{-10}$ \quad . 
\end{enumerate}
\end{lemma}
\begin{proof}
Left to the reader. The same comments as in the proof of 
Lemma~\ref{fdtr5eyrr5tytrew54ter456tr54r} apply. 
\end{proof}

\begin{lemma}[{uniform bounds for $B_2(t)$}]
\label{rew54trrte54ree3w54trwd}
With $B_2$ as in Definition~\ref{r54trerewr5w4ytrere5t4rw43re}.\ref{dfg54ter54trdr4e5wdfsew43re}, 
\begin{equation}\label{r34terew5454redf} 
-0.001491431277 < B_2(t) < -0.001491431155 \quad . 
\end{equation}
for every $t\in I:=(0.6263716633 - 10^{-10},0.6263716633 + 10^{-10})$.
\end{lemma}
\begin{proof}
The plan of the proof is the same as for Lemma~\ref{fdert34w43r4wer5t433453rte54tref}: 
for each of the four summands in $B_2$, 
derive both upper and lower bounds which uniformly hold in $I$. In the end, we add these 
bounds to derive the bounds in \eqref{r34terew5454redf}. 

In the following paragraph, we prove the uniform bounds 
\begin{equation}\label{fdswe54trrew54t54rtfgvcretrds}
0.01786492701 < \tfrac{(-t+1)^3(3t-1)(3t+1)(t+1)^3\log(t+1)}{256t^6} < 0.01786492706 \quad 
\parbox{0.15\linewidth}{for every $t\in I$} \quad . 
\end{equation}
Since $t\mapsto -1+3t^2$ is strictly monotone increasing in $I$, it follows that 
$-1+3t^2 > -1+3\cdot(0.6263716633 - 10^{-10})^2 = 0.17702438137980270272 > 0$ for every $t\in I$, 
and now it is evident from $\tfrac{\upd}{\upd t}$ $(-t+1)^3(3t-1)(3t+1)(t+1)^3$ $=$ 
$-24(-1+t)^2t(1+t)^2(-1+3t^2)$ that $\tfrac{\upd}{\upd t}\ (-t+1)^3(3t-1)(3t+1)(t+1)^3 < 0$ 
for every $t\in I$, i.e., that $t\mapsto (-t+1)^3(3t-1)(3t+1)(t+1)^3$ 
is strictly monotone decreasing in $I$, so 
\begin{align}\label{we34trer564tr43te5e534ter}
    (-t+1)^3(3t-1)(3t+1)(t+1)^3
& > (-t+1)^3(3t-1)(3t+1)(t+1)^3\biggr\rvert_{t=0.6263716633 + 10^{-10}} \notag \\
& > 0.56791522564 \quad\text{for every $t\in I$} \quad , 
\end{align}
and
\begin{align}\label{e2r5t4eq432rewwe24tr3454etrgf}
    (-t+1)^3(3t-1)(3t+1)(t+1)^3
& < (-t+1)^3(3t-1)(3t+1)(t+1)^3\biggr\rvert_{t=0.6263716633 - 10^{-10}} \notag \\
& < 0.56791522584 \quad\text{for every $t\in I$} \quad . 
\end{align}
Since $t\mapsto\log(t+1)$ is strictly monotone increasing in $I$, it follows that 
\begin{align}\label{dqwe432r4ewe4354rw54tr}
    \log(t+1)
& > \log(t+1)\biggr\rvert_{t=0.6263716633 - 10^{-10}} \notag \\
& > 0.48635156016\quad\text{for every $t\in I$} \quad , 
\end{align}
and
\begin{align}\label{fweet45e3w4tr43re5r4te}
    \log(t+1)
& < \log(t+1)\biggr\rvert_{t=0.6263716633 + 10^{-10}} \notag \\
& < 0.48635156029\quad\text{for every $t\in I$} \quad . 
\end{align}
Since $t\mapsto 256 t^6$ is strictly monotone increasing in $I$, it follows that 
\begin{align}\label{d45srt4twe5r4tewe54treer65tyer}
    256 t^6
& > 256 t^6\biggr\rvert_{t=0.6263716633 - 10^{-10}} > 15.46082193821\quad\text{for every $t\in I$} \quad , 
\end{align}
and
\begin{align}\label{s4e2w45435te54trre54tre4rte}
    256 t^6
& < 256 t^6\biggr\rvert_{t=0.6263716633 + 10^{-10}} < 15.46082196784\quad\text{for every $t\in I$} \quad . 
\end{align} 
It follows that, for every $t\in I$, 
\begin{align}\label{q432rew43re4w53r4ew43rew}
  & \tfrac{(-t+1)^3(3t-1)(3t+1)(t+1)^3\log(t+1)}{256t^6}\notag \\
\parbox{0.25\linewidth}{(by \eqref{we34trer564tr43te5e534ter})}
> & \tfrac{0.56791522564\cdot\log(t+1)}{256t^6} \notag \\
\parbox{0.25\linewidth}{(by \eqref{dqwe432r4ewe4354rw54tr})}
> & \tfrac{0.56791522564\cdot 0.48635156016}{256t^6} \notag \\
\parbox{0.25\linewidth}{(by \eqref{s4e2w45435te54trre54tre4rte})}
> & \tfrac{0.56791522564\cdot 0.48635156016}{15.46082196784} \notag \\
> & 0.01786492701\quad ,
\end{align}
proving the lower bound in \eqref{fdswe54trrew54t54rtfgvcretrds}, and also that,
for every $t\in I$, 
\begin{align}\label{dferdse5t4ererw54treeterfd}
  & \tfrac{(-t+1)^3(3t-1)(3t+1)(t+1)^3\log(t+1)}{256t^6}\notag \\
\parbox{0.25\linewidth}{(by \eqref{e2r5t4eq432rewwe24tr3454etrgf})}
< & \tfrac{0.56791522584\cdot\log(t+1)}{256t^6} \notag \\
\parbox{0.25\linewidth}{(by \eqref{fweet45e3w4tr43re5r4te})}
< & \tfrac{0.56791522584\cdot 0.48635156029}{256t^6} \notag \\
\parbox{0.25\linewidth}{(by \eqref{d45srt4twe5r4tewe54treer65tyer})}
< & \tfrac{0.56791522584\cdot 0.48635156029}{15.46082193821} \notag \\
< & 0.01786492706 \quad ,
\end{align}
proving the upper bound in \eqref{fdswe54trrew54t54rtfgvcretrds}. 

In the following paragraph, we prove the uniform bounds 
\begin{equation}\label{ew5r4terer5t465tr}
0.02019321732 < \tfrac{(-t+1)^3(3t+1)\log(3t+1)}{32t^3} < 0.02019321738 \quad 
\parbox{0.15\linewidth}{for every $t\in I$} \quad . 
\end{equation}
Since $\tfrac{\upd}{\upd t} (-t+1)^3(3t+1) = -12(-1+t)^2t < 0$ for every $t\in I$, 
we know that $t\mapsto (-t+1)^3(3t+1)$ is strictly monotone decreasing in $I$, hence 
\begin{align}\label{r54te454t3r65terqe545646yrt}
    (-t+1)^3(3t+1)
& > (-t+1)^3(3t+1)\biggr\rvert_{t=0.6263716633 + 10^{-10}} \notag \\
& > 0.15016835728\quad\text{for every $t\in I$} \quad , 
\end{align}
and
\begin{align}\label{wg5ye4ter42543654t64t5ref43}
    (-t+1)^3(3t+1)
& < (-t+1)^3(3t+1)\biggr\rvert_{t=0.6263716633 - 10^{-10}} \notag \\
& < 0.15016835750 \quad\text{for every $t\in I$} \quad . 
\end{align} 
Since $t\mapsto \log(3t+1)$ is strictly monotone increasing, it follows that for every $t\in I$, 
\begin{equation}\label{et65564tre455345e3465er}
\log(3t+1) > \log(3t+1)\biggr\rvert_{t=0.6263716633 - 10^{-10}} 
> 1.05748295164\quad\text{for every $t\in I$} \quad , 
\end{equation}
and 
\begin{equation}\label{fdret56eterdewre5436e4354}
\log(3t+1) < \log(3t+1)\biggr\rvert_{t=0.6263716633 + 10^{-10}}
< 1.05748295186\quad\text{for every $t\in I$} \quad.  
\end{equation}
Since $t\mapsto 32t^3$ is strictly monotone increasing in $I$, it follows that 
\begin{equation}\label{frw5e4terdsxw43554e}
32t^3 > 32t^3\biggr\rvert_{t=0.6263716632} > 7.86405034017 
\end{equation}
and 
\begin{equation}\label{fdgter54twwqw345443}
32t^3 < 32t^3\biggr\rvert_{t=0.6263716634} < 7.86405034771 \quad . 
\end{equation}
It follows that, for every $t\in I$, 
\begin{align}\label{ader43reew454rerw54354wrw45ew}
  & \tfrac{(-t+1)^3(3t+1)\log(3t+1)}{32t^3}\notag \\
\parbox{0.25\linewidth}{(by \eqref{r54te454t3r65terqe545646yrt})}
> & \tfrac{0.15016835728\cdot\log(3t+1)}{32t^3} \notag \\
\parbox{0.25\linewidth}{(by \eqref{et65564tre455345e3465er})}
> & \tfrac{0.15016835728\cdot 1.05748295164}{32t^3} \notag \\
\parbox{0.25\linewidth}{(by \eqref{fdgter54twwqw345443})}
> & \tfrac{0.15016835728\cdot 1.05748295164}{7.86405034771} \notag \\
> & 0.02019321732\quad ,
\end{align}
proving the lower bound in \eqref{ew5r4terer5t465tr}, and also that, 
for every $t\in I$, 
\begin{align}\label{q23454e2354546etr54re45}
  & \tfrac{(-t+1)^3(3t+1)\log(3t+1)}{32t^3}\notag \\
\parbox{0.25\linewidth}{(by \eqref{wg5ye4ter42543654t64t5ref43})}
< & \tfrac{0.15016835750\cdot\log(3t+1)}{32t^3} \notag \\
\parbox{0.25\linewidth}{(by \eqref{fdret56eterdewre5436e4354})}
< & \tfrac{0.15016835750\cdot 1.05748295186}{32t^3} \notag \\
\parbox{0.25\linewidth}{(by \eqref{frw5e4terdsxw43554e})}
< & \tfrac{0.15016835750\cdot 1.05748295186}{7.86405034017} \notag \\
< & 0.02019321738\quad ,
\end{align}
proving the upper bound in \eqref{ew5r4terer5t465tr}. 

In the following paragraph, we prove the uniform bounds 
\begin{equation}\label{asdw4ewrw5r4wr54re}
0.00059228380 < \tfrac{(3t+1)^2(-t+1)^6\log(2t+1)}{512t^6} < 0.00059228381  \quad 
\parbox{0.15\linewidth}{for every $t\in I$} \quad . 
\end{equation}
Since $\tfrac{\upd}{\upd t}$ $(3t+1)^2(-t+1)^6$ $=$ $24(t-1)^5t(3t+1)$ $<0$ 
for every $t\in I$, we know that $t\mapsto (3t+1)^2(-t+1)^6$ is strictly monotone 
decreasing in $I$, hence 
\begin{align}\label{wer54etwwee4353e4564tr}
    (3t+1)^2(-t+1)^6
& > (3t+1)^2(-t+1)^6\biggr\rvert_{t=0.6263716633 + 10^{-10}} 
> 0.02255053553\quad\text{for every $t\in I$} \quad , 
\end{align}
and
\begin{align}\label{egrt65re54tr654tr54tr4}
    (3t+1)^2(-t+1)^6
& < (3t+1)^2(-t+1)^6\biggr\rvert_{t=0.6263716633 - 10^{-10}} 
< 0.02255053560 \quad\text{for every $t\in I$} \quad . 
\end{align} 
Since $t\mapsto\log(2t+1)$ is strictly monotone increasing in $I$, it follows that 
\begin{equation}\label{rte54rre5er5e4trre5}
\log(2t+1) > \log(2t+1)\biggr\rvert_{t=0.6263716633 - 10^{-10}} > 0.81214872970 
\quad\text{for every $t\in I$}
\end{equation}
and 
\begin{equation}\label{fsgr4tweretew4534re}
\log(2t+1) < \log(2t+1)\biggr\rvert_{t=0.6263716633 + 10^{-10}} < 0.81214872989
\quad\text{for every $t\in I$}\quad . 
\end{equation}
Since $t\mapsto 512 t^6$ is strictly monotone increasing in $I$, it follows that 
\begin{equation}\label{fg5te4rew54trret6y5tre}
512 t^6 > 512 t^6\biggr\rvert_{t=0.6263716633 - 10^{-10}} > 30.92164387643
\quad\text{for every $t\in I$}
\end{equation}
and 
\begin{equation}\label{ert546r5tee54tr3e54e54}
512 t^6 < 512 t^6\biggr\rvert_{t=0.6263716633 + 10^{-10}} < 30.92164393568
\quad\text{for every $t\in I$}\quad . 
\end{equation}
It follows that, for every $t\in I$, 
\begin{align}\label{er534rere54te54rtw454tr}
  & \tfrac{(3t+1)^2(-t+1)^6\log(2t+1)}{512t^6}\notag \\
\parbox{0.25\linewidth}{(by \eqref{wer54etwwee4353e4564tr})}
> & \tfrac{0.02255053553\cdot\log(2t+1)}{512t^6} \notag \\
\parbox{0.25\linewidth}{(by \eqref{rte54rre5er5e4trre5})}
> & \tfrac{0.02255053553\cdot 0.81214872970}{512t^6} \notag \\
\parbox{0.25\linewidth}{(by \eqref{ert546r5tee54tr3e54e54})}
> & \tfrac{0.02255053553\cdot 0.81214872970}{30.92164393568} \notag \\
> & 0.00059228380\quad ,
\end{align}
proving the lower bound in \eqref{asdw4ewrw5r4wr54re}, and, for every $t\in I$, 
\begin{align}\label{we54trer54t56e5tre65tr}
  & \tfrac{(3t+1)^2(-t+1)^6\log(2t+1)}{512t^6}\notag \\
\parbox{0.25\linewidth}{(by \eqref{egrt65re54tr654tr54tr4})}
< & \tfrac{0.02255053560\cdot\log(2t+1)}{512t^6} \notag \\
\parbox{0.25\linewidth}{(by \eqref{fsgr4tweretew4534re})}
< & \tfrac{0.02255053560\cdot 0.81214872989}{512t^6} \notag \\
\parbox{0.25\linewidth}{(by \eqref{fg5te4rew54trret6y5tre})}
< & \tfrac{0.02255053560\cdot 0.81214872989}{30.92164387643} \notag \\
< & 0.00059228381\quad ,
\end{align}
proving the upper bound in \eqref{asdw4ewrw5r4wr54re}. 

In the following paragraph, we prove the uniform bounds 
\begin{equation}\label{ew54tee54rtw4ew354wes}
0.000244575293 < 
\tfrac{(t-1)^4(185t^4+698t^3-217t^2-160t+6)}{1024 t^4(3t+1)(t+3)} < 
0.000244575295 \quad \parbox{0.15\linewidth}{for every $t\in I$} \quad . 
\end{equation}
For every $t\in I$, evidently $(t-1)^3<0$. 
Moreover, since both $t\mapsto 740t^4+2073t^3+92$ and $t\mapsto 1698t^2+183t$ are strictly monotone 
increasing in $I$, we have, for every $t\in I$, 
\begin{align}\label{fret54trdere54te3254w354e}
740t^4+2073t^3+92 
& < 740t^4+2073t^3+92\biggr\rvert_{t=0.6263716633 + 10^{-10}} \notag \\
& = 715.3525597141428299499534408273089356632640 \notag \\
& < 780.82181422656832973952 \notag \\ 
& = 1698t^2+183t\biggr\rvert_{t=0.6263716633 - 10^{-10}} \notag \\ 
& < 1698t^2+183t \quad , 
\end{align}
i.e., $(740t^4+2073t^3-1698t^2-183t+92) < 0$ for every $t\in I$. 
Taken together, it follows that 
$\tfrac{\upd}{\upd t} (t-1)^4(185t^4+698t^3-217t^2-160t+6)$ $=$ 
$2(t-1)^3(740t^4+2073t^3-1698t^2-183t+92)$ $>$ $0$ for every $t\in I$, 
hence $t\mapsto (t-1)^4(185t^4+698t^3-217t^2-160t+6)$ 
is strictly monotone increasing in $I$, so 
\begin{align}\label{ew5t4terwr54te645tr}
   & (t-1)^4(185t^4+698t^3-217t^2-160t+6)   \notag \\
 > & (t-1)^4(185t^4+698t^3-217t^2-160t+6)\biggr\rvert_{t=0.6263716633 - 10^{-10}} \notag \\
 > & 0.40250592053\quad\text{for every $t\in I$} \quad , 
\end{align}
and
\begin{align}\label{rtee65tryer5465t6545}
   & (t-1)^4(185t^4+698t^3-217t^2-160t+6)   \notag \\    
 < & (t-1)^4(185t^4+698t^3-217t^2-160t+6)\biggr\rvert_{t=0.6263716633 + 10^{-10}} \notag \\
 < & 0.40250592191\quad\text{for every $t\in I$} \quad . 
\end{align} 
Since $t\mapsto 1024 t^4(3t+1)(t+3)$ is strictly monotone increasing, we furthermore know 
\begin{align}\label{tr465trr456e443e4trefgds}
   & 1024 t^4(3t+1)(t+3) \notag \\
 > & 1024 t^4(3t+1)(t+3)\biggr\rvert_{t=0.6263716633 - 10^{-10}} \notag \\
 > & 1645.7341777\quad\text{for every $t\in I$} \quad , 
\end{align}
and
\begin{align}\label{wqe3w45etre43rew43rewewt}
   & 1024 t^4(3t+1)(t+3) \notag \\    
 < & 1024 t^4(3t+1)(t+3)\biggr\rvert_{t=0.6263716633 + 10^{-10}} \notag \\
 < & 1645.7341803\quad\text{for every $t\in I$} \quad . 
\end{align} 

It follows that, for every $t\in I$, 
\begin{align}\label{54t5rw54354wr45t4retqw435ew}
  & \tfrac{(t-1)^4(185t^4+698t^3-217t^2-160t+6)}{1024 t^4(3t+1)(t+3)}\notag \\
\parbox{0.25\linewidth}{(by \eqref{ew5t4terwr54te645tr})}
> & \tfrac{0.40250592053}{1024 t^4(3t+1)(t+3)} \notag \\
\parbox{0.25\linewidth}{(by \eqref{wqe3w45etre43rew43rewewt})}
> & \tfrac{0.40250592053}{1645.7341803} \notag \\
> & 0.000244575293\quad ,
\end{align}
proving the lower bound in \eqref{ew54tee54rtw4ew354wes}, and, for every $t\in I$, 
\begin{align}\label{wee24rew343re3w435r4erw453r}
  & \tfrac{(t-1)^4(185t^4+698t^3-217t^2-160t+6)}{1024 t^4(3t+1)(t+3)}\notag \\
\parbox{0.25\linewidth}{(by \eqref{rtee65tryer5465t6545})}
< & \tfrac{0.40250592191}{1024 t^4(3t+1)(t+3)} \notag \\
\parbox{0.25\linewidth}{(by \eqref{tr465trr456e443e4trefgds})}
< & \tfrac{0.40250592191}{1645.7341777} \notag \\
< & 0.000244575295\quad ,
\end{align}
proving the upper bound in \eqref{ew54tee54rtw4ew354wes}.

From the lower bound in \eqref{fdswe54trrew54t54rtfgvcretrds}, 
the upper bound in \eqref{ew5r4terer5t465tr}, and 
the lower bounds in \eqref{asdw4ewrw5r4wr54re} and \eqref{ew54tee54rtw4ew354wes}, 
it follows that, for every $t\in I$, 
\begin{align}\label{derw54tresxwq4t3rewrw5e4trd}
B_2(t) & > 0.01786492701 - 0.02019321738 + 0.00059228380 + 0.000244575293 \notag \\ 
       & = -0.001491431277 \quad , 
\end{align}
proving the lower bound in \eqref{r34terew5454redf}. 

From the upper bound in \eqref{fdswe54trrew54t54rtfgvcretrds}, 
the lower bound in \eqref{ew5r4terer5t465tr}, and 
the upper bounds in \eqref{asdw4ewrw5r4wr54re} and \eqref{ew54tee54rtw4ew354wes}, 
it follows that, for every $t\in I$, 
\begin{align}\label{tr54tyrew43rerw5teew5t4rer5}
B_2(t) & < 0.01786492706 - 0.02019321732 + 0.00059228381 + 0.000244575295 \notag \\ 
       & =  -0.001491431155 \quad , 
\end{align}
proving the upper bound in \eqref{r34terew5454redf}. 
This completes the proof of Lemma~\ref{rew54trrte54ree3w54trwd}. 
\end{proof}

\begin{lemma}[{bounds for $B_2(t_0)$}]
\label{r54ter534re4654ter5434re}
With $B_2$ as in 
Definition~\ref{r54trerewr5w4ytrere5t4rw43re}.\ref{r435terew534w254fe43rew}, 
\begin{equation}\label{fwr554er4trgfewr5etew54rewds}
-0.001491431277 < B_2(t_0) < -0.001491431155
\end{equation}
\end{lemma}
\begin{proof}
In view of Lemma~\ref{t54treert4re354tree34terfd}, 
the bounds in \eqref{fwr554er4trgfewr5etew54rewds} follow 
from the uniform bounds in Lemma~\ref{rew54trrte54ree3w54trwd}.
\end{proof}

\begin{lemma}\label{tre54tr54treg53refdrefdw43erd}
\begin{equation}\label{e54trf54tre54tree5t4r}
0.0381910976 = 0.0381910977 - 10^{-10} < R(1) < 0.0381910976 + 10^{-10} < 0.0381910977 \quad . 
\end{equation}
\end{lemma}
\begin{proof}
By Definition~\ref{ftr5y4terwr454et54tre43ew354er54rted}, we know that 
with $t_0$ as in Lemma~\ref{rew54terer5etwe5redew54re} we have 
$R(1)$ $=$ $\tfrac{\left(3\cdot t_0 + 1\right)\left(-t_0 + 1\right)^3}{16\cdot t_0^3}$. 
It is routine to check that the function 
$t\mapsto \xi(t):=\tfrac{\left(3\cdot t + 1\right)\left(-t + 1\right)^3}{16\cdot t^3}$ 
is strictly monotone decreasing for $t\in (0,1)$, hence $R(1)=\xi(t_0)$ together with 
the bounds on $t_0$ from \eqref{fregtyess5rtyewe54retfews5t} in 
Lemma~\ref{t54treert4re354tree34terfd} implies  
\begin{equation}\label{ert46trer54trw4534re54tr54tr54tr} 
\xi(0.6263716633 - 10^{-10}) < R(1) < \xi(0.6263716633 + 10^{-10}) \quad ,
\end{equation}
so in \eqref{e54trf54tre54tree5t4r} the lower bound follows from the lower bound 
in Lemma~\ref{fdg54treert3254rede54terdrw45r657ty}.\ref{sdwferf54trre5665r45tere64tr}, 
while the upper bound follows from the upper bound in 
Lemma~\ref{fdg54treert3254rede54terdrw45r657ty}.\ref{fdgwret4rgddewr54etrderr54te}. 
\end{proof}

\begin{lemma}[{exact formula for $\nu$ in terms of $t_0$}]
\label{gr5t34terre5t4564w254} 
With $\rho=\gamma^{-1}$ as in \cite[p.~310]{GimenezNoy}, 
$C$ the exponential generating function of connected labelled planar graphs, 
and with $B_0$ and $B_2$ as in Definition~\ref{r54trerewr5w4ytrere5t4rw43re}, 
and with $R$ as in \cite[(2.6)]{GimenezNoy} and $B_0$ and $B_2$ as in 
Definition~\ref{r54trerewr5w4ytrere5t4rw43re}, 
\begin{equation}\label{rew564treew453r54trdgfrwre54e}
\nu := C(\rho) = R(1) + B_0(t_0) + B_2(t_0)\quad . 
\end{equation}
\end{lemma}
\begin{proof}
See \cite[p.~321,~(4.7)]{GimenezNoy}, together with the equation immediately above that. 
\end{proof}

\begin{lemma}[{verified bounds for $\nu$}]
\label{54re4565rt43refd54ert5645344545terdfx545454}
The real number $\nu$ defined in \cite{GimenezNoy} satisfies 
\begin{equation}\label{rew45tre54tre54tredswe43eds}
0.037439365283 < \nu < 0.037439366735\quad . 
\end{equation}
\end{lemma}
\begin{proof}
The lower bound follows from 
\begin{align}\label{defwr5tr4twsrewt5r4etewrewfd}
\nu & \By{\eqref{rew564treew453r54trdgfrwre54e}}{=}  R(1) + B_0(t_0) + B_2(t_0) \notag \\
\parbox{0.3\linewidth}{(by Lemmas~\ref{tre54tr54treg53refdrefdw43erd}, 
\ref{fdrte54rdere4rewr544rt} and \ref{r54ter534re4654ter5434re})} 
& > 0.0381910976 + 0.00073969896 + (-0.001491431277) \notag \\
& = 0.037439365283
\end{align}
and the upper bound from 
\begin{align}\label{rty54tr4w35re53r4er65w4terfd}
\nu & \By{\eqref{rew564treew453r54trdgfrwre54e}}{=}  R(1) + B_0(t_0) + B_2(t_0) \notag \\
\parbox{0.3\linewidth}{(by Lemmas~\ref{tre54tr54treg53refdrefdw43erd}, 
\ref{fdrte54rdere4rewr544rt} and \ref{r54ter534re4654ter5434re})} 
& < 0.0381910977 + 0.00073970019 + (-0.001491431155) \notag \\
& = 0.037439366735 \quad .  
\end{align}
\end{proof}

\begin{definition}[{$A(t)$, $\rho(t)$}]
\label{dfter54trrw54tt56356rt} 
With 
\begin{align}\label{fdret54rswe5r4we5t4rew54terwe543re}
A(t) := \tfrac{(3t-1)(t+1)^3\log(t+1)}{16t^3} + \tfrac{(3t+1)(-t+1)^3\log(2t+1)}{32t^3} + 
\tfrac{(-t+1)(185t^4+698t^3-217t^2-160t+6)}{64t(3t+1)^2(t+3)}
\end{align}
and
\begin{equation}\label{re54terdfrw45r4wrw54rewr43rew453re}
r(t) := \tfrac{1}{16}(3t+1)^{\frac12}(t^{-1}-1)^3\exp(A(t))
\end{equation}
we define
\begin{equation}\label{we54tyrerew5reswr5treew5red}
\rho := r(t_0)
\end{equation}
\end{definition}
\begin{proof}
See \cite[p.~310]{GimenezNoy}. 
\end{proof}

\begin{lemma}[{uniform bounds for $A(t)$}]
\label{rt5etr3e54tred}
With $A$ as in Definition~\ref{dfter54trrw54tt56356rt}, 
\begin{equation}\label{retw54tre45re5465} 
0.48968967248 < A(t) < 0.48968967363 
\end{equation}
for every $t\in I:=(0.6263716633 - 10^{-10},0.6263716633 + 10^{-10})$.
\end{lemma}
\begin{proof}
The structure of the proof is analogous to the proofs of 
Lemmas~\ref{fdert34w43r4wer5t433453rte54tref} and \ref{rew54trrte54ree3w54trwd}. 

In the following paragraph, we prove the uniform bounds 
\begin{equation}\label{fdwter54rerew5r4ewered}
0.46777725975 < \tfrac{(3t-1)(t+1)^3\log(t+1)}{16t^3} < 0.46777726082 \quad 
\text{for every $t\in I$} \quad . 
\end{equation}
Because of $\tfrac{\upd}{\upd t}\ (3t-1)(t+1)^3$ $=$ $12t(t+1)^2$ $>$ $0$ for every $t\in I$, 
we know that $t\mapsto (3t-1)(t+1)^3$ is strictly monotone increasing in $I$ and hence 
\begin{align}\label{r54trwr54t3e3425ter334e54trgfdscx}
    (3t-1)(t+1)^3 & > (3t-1)(t+1)^3\biggr\rvert_{t=0.6263716632} 
  > 3.78185681259 \quad\text{for every $t\in I$} \quad , 
\end{align}
and
\begin{align}\label{354tr4erttyrgd4ewrdsrefd54terer}
    (3t-1)(t+1)^3 & < (3t-1)(t+1)^3\biggr\rvert_{t=0.6263716634} 
  < 3.78185681657\quad\text{for every $t\in I$} \quad . 
\end{align} 
Since $t\mapsto 16 t^3$ is strictly monotone increasing, 
\begin{align}\label{5t4rge544ret32w54wrew5}
    16t^3 & > 16t^3\biggr\rvert_{t=0.6263716632} 
  > 3.93202517008 \quad\text{for every $t\in I$} \quad , 
\end{align}
and
\begin{align}\label{er54trefte5645trgfc}
    16t^3 & < 16t^3\biggr\rvert_{t=0.6263716634} 
  < 3.93202517386 \quad\text{for every $t\in I$} \quad . 
\end{align} 
From \eqref{r54trwr54t3e3425ter334e54trgfdscx}, 
\eqref{dqwe432r4ewe4354rw54tr} and \eqref{er54trefte5645trgfc} follows 
$((3t-1)(t+1)^3\cdot\log(t+1))/(16t^3)$ $>$ $3.78185681259$ $\cdot$ $0.48635156016$ $/$ 
$3.93202517386$ $>$ $0.46777725975$ for every $t\in I$, 
proving the lower bound in \eqref{fdwter54rerew5r4ewered}. 
From \eqref{354tr4erttyrgd4ewrdsrefd54terer}, 
\eqref{fweet45e3w4tr43re5r4te} and \eqref{5t4rge544ret32w54wrew5} follows 
$((3t-1)(t+1)^3\cdot\log(t+1))/(16t^3)$ $<$ $3.78185681657$ $\cdot$ $0.48635156029$ $/$ 
$3.93202517008$ $<$ $0.46777726082$ for every $t\in I$, 
proving the upper bound in \eqref{fdwter54rerew5r4ewered}. 

In the following paragraph, we prove the uniform bounds 
\begin{equation}\label{4e3wrer544red}
0.01550842571 < \tfrac{(3t+1)(-t+1)^3\log(2t+1)}{32t^3} < 0.01550842575 \quad 
\text{for every $t\in I$} \quad . 
\end{equation}

Since $\tfrac{\upd}{\upd t}\ (3t+1)(-t+1)^3$ $=$ $-12 t (t-1)^2$ $<$ $0$ for every $t\in I$, 
we know that $t\mapsto (3t+1)(-t+1)^3$ is strictly monotone decreasing in $I$, so 
\begin{align}\label{43wrte453534654rewda}
    (3t+1)(-t+1)^3 & > (3t+1)(-t+1)^3\biggr\rvert_{t=0.6263716634} 
  > 0.15016835728 \quad\text{for every $t\in I$} \quad , 
\end{align}
and
\begin{align}\label{w43wrer454wr544re4rw}
    (3t+1)(-t+1)^3 & < (3t+1)(-t+1)^3\biggr\rvert_{t=0.6263716632} 
  < 0.15016835750 \quad\text{for every $t\in I$} \quad . 
\end{align} 
From \eqref{43wrte453534654rewda}, \eqref{dew54retw54trt5e4tyrgfd} 
and \eqref{fdgter54twwqw345443} it follows that 
$((3t+1)(-t+1)^3\cdot\log(2t+1))/(32t^3)$ $>$ $0.15016835728$ $\cdot$ $0.81214872970$ $/$ 
$7.86405034771$ $>$ $0.01550842571$, proving the lower bound in \eqref{4e3wrer544red}.
From \eqref{w43wrer454wr544re4rw}, \eqref{d5645tw54te54erte5w4tref} 
and \eqref{frw5e4terdsxw43554e} it follows that 
$((3t+1)(-t+1)^3\cdot\log(2t+1))/(32t^3)$ $<$ $0.15016835750$ $\cdot$ $0.81214872989$ $/$ 
$7.86405034017$ $<$ $0.01550842575$. 

In the following paragraph, we prove the uniform bounds 
\begin{equation}\label{erw54treer4ew5tre45terfds}
0.00640398702 
< \tfrac{(-t+1)(185t^4+698t^3-217t^2-160t+6)}{64t(3t+1)^2(t+3)} < 
0.00640398706 \quad \text{for every $t\in I$} \quad . 
\end{equation}
Since both $t\mapsto 2745t^2$ and $t\mapsto 925t^4 + 2052t^3 + 114t + 166$ are strictly 
monotone increasing in $I$, we have 
$2745t^2$ $>$ $2745t^2\biggr\rvert_{t=0.6263716632}$ $>$ 
$1076.97730896$ $>$ $884.07553334$ $>$ 
$925t^4 + 2052t^3 + 114t + 166\biggr\rvert_{t=0.6263716634}$ $>$ $925t^4 + 2052t^3 + 114t + 166$ 
for every $t\in I$, hence $\tfrac{\upd}{\upd t} (-t+1)(185t^4+698t^3-217t^2-160t+6)$ $=$ 
$-925t^4-2052t^3+2745t^2-114t-166$ $>$ $0$ for every $t\in I$. Therefore, 
$t\mapsto (-t+1)(185t^4+698t^3-217t^2-160t+6)$ is strictly monotone increasing in $I$, so 
\begin{align}\label{ewr54treedew5r4tewe43reds}
  &  (-t+1)(185t^4+698t^3-217t^2-160t+6) \notag \\
> & (-t+1)(185t^4+698t^3-217t^2-160t+6)\biggr\rvert_{t=0.6263716632} \notag \\
> & 7.71707734263 \quad\text{for every $t\in I$} \quad , 
\end{align}
and
\begin{align}\label{erw5se435e6r76uythg}
  &  (-t+1)(185t^4+698t^3-217t^2-160t+6) \notag \\  
< & (-t+1)(185t^4+698t^3-217t^2-160t+6)\biggr\rvert_{t=0.6263716634} \notag \\
< & 7.71707738122 \quad\text{for every $t\in I$} \quad . 
\end{align} 
Since $t\mapsto 64t(3t+1)^2(t+3)$ is strictly monotone increasing in $I$, we have 
\begin{align}\label{er54we32e5t3wrewe43errew53r4ew53r4e}
  64t(3t+1)^2(t+3) > 64t(3t+1)^2(t+3)\biggr\rvert_{t=0.6263716633 - 10^{-10}} 
> 1205.0426269\quad\text{for every $t\in I$} \quad , 
\end{align}
and
\begin{align}\label{srteewsdterdgsdrefdsrtefdsfdtrdgste}
  64t(3t+1)^2(t+3) < 64t(3t+1)^2(t+3)\biggr\rvert_{t=0.6263716633 + 10^{-10}} 
< 1205.0426279\quad\text{for every $t\in I$} \quad . 
\end{align} 
From \eqref{ewr54treedew5r4tewe43reds} and \eqref{srteewsdterdgsdrefdsrtefdsfdtrdgste} 
it follows that $(-t+1)(185t^4+698t^3-217t^2-160t+6)/(64t(3t+1)^2(t+3))$ $>$ 
$7.71707734263$ $/$ $1205.0426279$ $>$ $0.00640398702$, proving the 
lower bound in \eqref{erw54treer4ew5tre45terfds}. From \eqref{erw5se435e6r76uythg} 
and \eqref{er54we32e5t3wrewe43errew53r4ew53r4e} it follows that 
$(-t+1)(185t^4+698t^3-217t^2-160t+6)/(64t(3t+1)^2(t+3))$ $<$  
$7.71707738122$ $/$ $1205.0426269$ $<$ $0.00640398706$, 
proving the upper bound in \eqref{erw54treer4ew5tre45terfds}.

In view of Definition~\ref{dfter54trrw54tt56356rt}, the lower bounds in 
\eqref{fdwter54rerew5r4ewered}, \eqref{4e3wrer544red} and \eqref{erw54treer4ew5tre45terfds} 
imply that for every $t\in I$, 
\begin{equation}\label{r54yertfdswrewerewsed}
A(t) > 0.46777725975 + 0.01550842571 + 0.00640398702 = 0.48968967248\quad , 
\end{equation}
proving the lower bound in \eqref{retw54tre45re5465}, while the upper bounds in 
\eqref{fdwter54rerew5r4ewered}, \eqref{4e3wrer544red} and \eqref{erw54treer4ew5tre45terfds} imply 
that for every $t\in I$, 
\begin{equation}\label{56465tret645terf}
A(t) < 0.46777726082 + 0.01550842575 + 0.00640398706 = 0.48968967363 \quad , 
\end{equation}
proving the upper bound in \eqref{retw54tre45re5465}. 
\end{proof}

\begin{lemma}[{bounds for $A(t_0)$}]\label{fwer54trgfrew54r4ew2ed}
With $A(t)$ as in Definition~\ref{dfter54trrw54tt56356rt} and 
$t_0$ as in Definition~\ref{t54treert4re354tree34terfd}, 
\begin{equation}\label{frte4rre5t4re564t5reretr} 
0.48968967248 < A(t_0) < 0.48968967363 \quad . 
\end{equation}
\end{lemma}
\begin{proof}
Immediate from Lemmas~\ref{t54treert4re354tree34terfd} and \ref{rt5etr3e54tred}. 
\end{proof}

\begin{lemma}[{uniform bounds for $r(t)$}]
\label{ert54yrdfew4544r5etr3e4tewrfds}
With $r(t)$ as in Definition~\ref{dfter54trrw54tt56356rt} 
and $I$ as in Lemma~\ref{fdert34w43r4wer5t433453rte54tref}, 
\begin{equation}\label{re54trwrr54terr3454etr4354rte}
0.03672841251 < r(t) < 0.03672841266\quad \text{for every $t\in I$}\quad . 
\end{equation}
\end{lemma}
\begin{proof}
Since $t\mapsto\tfrac{1}{16}(3t+1)^{\frac12}$ is evidently strictly monotone increasing in $I$, 
\begin{equation}\label{rt5e46trere5t4rwe54terwds}
  \tfrac{1}{16}(3t+1)^{\frac12} 
> \tfrac{1}{16}(3t+1)^{\frac12}\biggr\rvert_{t=0.6263716633 - 10^{-10}} 
> 0.10604971913\quad\text{for every $t\in I$}
\end{equation}
and 
\begin{equation}\label{try465treer454rteret54tree54ree}
  \tfrac{1}{16}(3t+1)^{\frac12} 
< \tfrac{1}{16}(3t+1)^{\frac12}\biggr\rvert_{t=0.6263716633 + 10^{-10}} 
< 0.10604971915\quad\text{for every $t\in I$} \quad .  
\end{equation}

Since $t\mapsto t^{-1}-1$ is strictly monotone decreasing in $I$, so is $t\mapsto(t^{-1}-1)^3$, 
hence 
\begin{equation}\label{fdgrte5464e4535454ew}
   (t^{-1}-1)^3 > (t^{-1}-1)^3\biggr\rvert_{t=0.6263716633 + 10^{-10}} 
               > 0.21223798428 \quad\text{for every $t\in I$}
\end{equation}
and 
\begin{equation}\label{fert56terdew5454tr564t5er5}
   (t^{-1}-1)^3 < (t^{-1}-1)^3\biggr\rvert_{t=0.6263716633 - 10^{-10}} 
               < 0.21223798483\quad\text{for every $t\in I$} \quad . 
\end{equation}

Combining Lemma~\ref{fwer54trgfrew54r4ew2ed} with
 \ref{dft546rte54ee54564544}.\ref{fter46tyr54trer54tr54tr54ter} 
in Lemma~\ref{rw5t4rswre54trt546tr54trde}, and since $\exp$ is strictly monotone increasing, 
it follows that, for every $t\in I$, 
\begin{align}\label{fre54tergsdre5terrte5465454535terf}
\exp(A(t)) & > \exp(0.48968967248) \notag \\
           & > 0.39995948109\cdot 10^{-12} + \sum_{0\leq i \leq 11} (0.48968967248)^i/i! \notag \\
           & > 1.63180974590 \quad ,  
\end{align}
and, again for every $t\in I$, 
\begin{align}\label{fre5t465terrt54t5654654tyr746ytr5t}
\exp(A(t)) & < \exp(0.48968967363) \notag \\
           & < 0.11998784433\cdot 10^{-11} + \sum_{0\leq i \leq 11} (0.48968967363)^i/i! \notag \\
           & < 1.63180974778 \quad . 
\end{align}
It follows that, for every $t\in I$, 
\begin{align}\label{5trrw54t5tt5645t43r54retdf} 
r(t) & = \tfrac{1}{16}(3t+1)^{\frac12}(t^{-1}-1)^3\exp(A(t)) \notag \\
\parbox{0.1\linewidth}{\eqref{rt5e46trere5t4rwe54terwds}} & >  
0.10604971913\cdot(t^{-1}-1)^3\exp(A(t)) \notag \\
 \parbox{0.1\linewidth}{\eqref{fdgrte5464e4535454ew}} & >  
0.10604971913\cdot 0.21223798428\cdot\exp(A(t)) \notag \\ 
 \parbox{0.1\linewidth}{\eqref{fre54tergsdre5terrte5465454535terf}} & >  
0.10604971913\cdot 0.21223798428\cdot 1.63180974590 \notag \\ 
& > 0.03672841251\quad ,
\end{align}
proving the lower bound in \eqref{re54trwrr54terr3454etr4354rte}, and, for every $t\in I$, 
\begin{align}\label{dfer5y4tr4t354re4rt5r4twe5ter} 
r(t) & = \tfrac{1}{16}(3t+1)^{\frac12}(t^{-1}-1)^3\exp(A(t)) \notag \\
\parbox{0.1\linewidth}{\eqref{try465treer454rteret54tree54ree}} & <  
0.10604971915\cdot(t^{-1}-1)^3\exp(A(t)) \notag \\
 \parbox{0.1\linewidth}{\eqref{fert56terdew5454tr564t5er5}} & <  
0.10604971915\cdot 0.21223798483\cdot\exp(A(t)) \notag \\ 
 \parbox{0.1\linewidth}{\eqref{fre5t465terrt54t5654654tyr746ytr5t}} & <  
0.10604971915\cdot 0.21223798483\cdot 1.63180974778 \notag \\ 
& < 0.03672841266  \quad ,
\end{align}
proving the upper bound in \eqref{re54trwrr54terr3454etr4354rte}. 
\end{proof}

\begin{proofof}{Lemma~\ref{lem:approx}}
Since $\rho=r(t_0)$ by Definition~\ref{dfter54trrw54tt56356rt}, it is immediate from 
Lemmas~\ref{t54treert4re354tree34terfd} and \ref{ert54yrdfew4544r5etr3e4tewrfds} that 
\[ 0.03672841251 < \rho < 0.03672841266. \]
\noindent
Recall that $G(\rho) = \exp(C(\rho)) = \exp(\nu)$.
According to Lemma~\ref{54re4565rt43refd54ert5645344545terdfx545454} 
we have $-0.037439366735 < -\nu < -0.037439365283$, 
so \ref{w54tryfe54t565tr54trd}.\ref{fdgwr54yterew43rew435re543rds} 
in Lemma~\ref{rw5t4rswre54trt546tr54trde} is applicable and implies, 
by strict monotonicity of $\exp$, 
\[ \exp(-\nu) > 
1.0850694444\cdot 10^{-11} + \sum_{0\leq i\leq 5} \tfrac{(-0.037439366735)^i}{i!} > 0.96325282112 \]
and 
\[
\exp(-\nu)  < 
2.1701388889\cdot 10^{-11} + \sum_{0\leq i\leq 5} \tfrac{(-0.037439365283)^i}{i!} < 0.96325282254 \quad . \]
\end{proofof}

\bibliographystyle{plain}
\bibliography{ReferencesMSO_Tobias_Final}

\end{document}